\documentclass[reqno]{article}
\usepackage[utf8]{inputenc}

\title{Catalan Numbers and Jacobi Polynomials}
\author{Thomas M. Richardson}

\usepackage{amsmath}
\usepackage{amsthm}
\usepackage[colorlinks=true]{hyperref}
\ifdefined\reusenumber
    \Copyright{Thomas M Richardson. Released under the CC BY-ND license (International 4.0).}
    \dateline{May 18, 2020}{July 20, 2020}{July 21, 2020}
    \MSC{05A10, 15B36}
\else
    \date{August 24, 2020}
    \newtheorem{theorem}{Theorem}[section]
    \newtheorem{lemma}{Lemma}[section]
    \newtheorem{corollary}{Corollary}[section]
    
    \theoremstyle{definition}
    \newtheorem{definition}{Definition}[section]
\fi

\newcommand{\seqnum}[1]{\href{https://oeis.org/#1}{\underline{#1}}}

\begin{document}

\maketitle

\begin{abstract}
We prove that the inverse of the Hankel matrix of 
 the reciprocals of the Catalan numbers has integer entries. We generalize the result to an infinite family of generalized Catalan numbers. The Hankel matrices that we consider are associated with orthogonal polynomials that are variants of Jacobi polynomials. Our proofs use these polynomials and computer algebra based on Wilf-Zeilberger theory.
\end{abstract}

\section{Introduction}
We prove the that the inverse of the Hankel matrix of the reciprocals of the Catalan numbers has integer entries. We prove the same result for an infinite family of generalized Catalan numbers. Based on the analogy of these matrices with the Hilbert matrix, we call the Hankel matrix of reciprocals of the Catalan numbers the {\it Catbert matrix}, a portmanteau of Catalan and Hilbert matrix. 

All matrices in this paper are indexed starting with 0, and are infinite unless otherwise stated. For an infinite matrix $M$, denote the $n\times n$ upper left submatrix by $M(n)$.

The original inspiration for this work is the the Hilbert matrix $H$, defined by $H=[ h_{i,j} ]$ with $h_{i,j}=\frac{1}{i+j+1}$, and the interesting fact that the inverse of the $n \times n$ Hilbert matrix has integer entries. Choi asked what sort of coincidence it is if the inverse of a matrix of reciprocals of integers has integer entries \cite{TRICKS}. The most general collection of sequences for which we prove that the corresponding Hankel matrices have integer inverses are sequences with generating functions of the form $g(x) = q(1-p^2x)^{q/p},$  where $p$ is an integer with $p\ge 2$, and $q$ is an integer relatively prime to $p$. The proof extends to their subsequences $(g_a, g_{a+1}, g_{a+2}, \ldots).$ 

Other matrices of reciprocals of integers that have been shown to have inverses with integer entries include the Filbert matrix and generalizations, based on reciprocals of Fibonacci numbers and generalized Fibonacci numbers \cite{FILBERT,KILICPRODINGER}; matrices based on reciprocals of sequences of binomial coefficients $b_n=\frac{1}{\alpha\binom{n+\alpha}{\alpha}}$ for a natural number $\alpha$ \cite{BERG}; and the matrices of reciprocals of the super Catalan and super Patalan numbers \cite{SUPERPAT_PUBLISHED}. 

Approaches to proving that the inverse matrices of reciprocals of integers have integer entries include orthogonal polynomials \cite{BERG}, and Wilf-Zeilberger (WZ) theory \cite{AeqB,FILBERT}. This paper will use both of these approaches in its proofs.

\section{Generalized Catbert Matrices}
First we define the sequences that we will show give rise to Hankel matrices of reciprocals of integers, whose inverses have integer entries.

\begin{definition}
\label{GENPATSEQ}
Let $p$ be an integer with $p\ge2$, and let $q$ be an integer that is relatively prime to $p$. Define the sequence of generalized Catalan numbers $g$ by 
\(g_n = p^{2n}\binom{n+q/p}{n}=(-p^2)^{n}\binom{-1-q/p}{n}\). 
We use $g^{(q/p)}=g$ when we want to be explicit about the values of the parameters $p$ and $q$.
\end{definition}
The sequences $g$ can be found as the non-zero entries of columns in a recursive matrix of super Catalan or super Patalan numbers \cite{LMMS, SUPERPAT_PUBLISHED}. The sequence of central binomial coefficients is $g^{(-1/2)}$ in this notation. The sequence of Catalan numbers is related to the sequence $g^{(-3/2)}$ as follows; if $g=g^{(-3/2)}$ and $C$ is the sequence of Catalan numbers, then $C_n = -\frac{1}{2}g_{n+1}.$

Theorems 9-11 of \cite{SUPERPAT_PUBLISHED} prove that the elements of $g$ are integers, with the qualification that the notation here is different, and the cited proofs only explicitly apply to a subset of the sequences that we consider here. The approach extends easily to the current context.

\begin{theorem}
The elements of the sequences $g^{(q/p)}$ defined in Definition \ref{GENPATSEQ} are integers.
\end{theorem}

\begin{proof}
Let $p$ be as in definition \ref{GENPATSEQ}, let $q=-(p+1)$, and let $g=g^{(q/p)}$. Now define the sequence $a$ for $n\ge 0$ by $a_n=-\frac{1}{p}g_{n+1}$. Then by Theorem 9 of \cite{SUPERPAT_PUBLISHED}, the sequence $a$ satisfies the recurrence relation
\[ a_n = \sum_{k=2}^p(-p)^{k-2}\binom{p}{k} \prod_{i_1+\ldots+i_k=n-k+1} a(i_j).
\]
Since $a_0=1$, this implies that $a_n$ is an integer, and therefore the elements of $g$ are integers, as $g_0=1$ and $g_n=-pa_{n-1}$.

Similarly, Theorem 10 of \cite{SUPERPAT_PUBLISHED} implies that the elements of $g^{(q/p)}$ are integers for $q=-(p+r)$ and $1<r<p$, as they are convolutional powers of $g^{(-(p+1)/p)}$.

Finally, the well known identity
\[\binom{n}{k}=\binom{n-1}{k}+\binom{n-1}{k-1} \]
implies the equations
\begin{align}
\label{GINT1}
g^{(q/p)}_n &= g^{((q-p)/p)}_n +p^2g^{(q/p)}_{n-1} \\
\label{GINT2}
g^{(q/p)}_n &= g^{((q+p)/p)}_n-p^2g^{((q+p)/p)}_{n-1}.
\end{align}
Equations \eqref{GINT1} and \eqref{GINT2} allow us to extend the proof that the elements of $g^{(q/p)}$ are integers to all $p,q$ satisfying the conditions in Definition \ref{GENPATSEQ}.
\end{proof}

Now we define the Hankel matrices based on generalized Catalan numbers. We define these Hankel matrices based on the subsequences of such a sequence $g$, starting with the term $g_a$ where $a\ge 0$.
\begin{definition}
\label{GENPATHANKEL}
Let $g$ be a sequence of generalized Catalan numbers, and let $a$ be a non-negative integer. Let the matrix $G$ be the Hankel matrix of reciprocals of $g$, offset by $a$, so that  $G_{i,j} = 1/g_{i+j+a}$, for $0\le i, j$. We use $G(n)$ to denote the $n \times n$ upper left submatrix of $G$. We also use $G^{(a,q/p)}$ and $G^{(a,q/p)}(n)$ for $G$ and $G(n)$, respectively, when we want to make the parameters $a$, $p$, and $q$ explicit.
\end{definition}

The sequences $g^{(q/p)}$ are based on sequences of binomial coefficients $\binom{n+q/p}{n}$. Berg described shifted Jacobi polynomials that are orthogonal polynomials for the Hankel matrix with entries $ \displaystyle{ \frac{1}{\alpha\binom{i+j+\alpha}{\alpha}}}$, where $\alpha$ is a natural number \cite[Theorem 4.2]{BERG}. He used the relationship to express the inverse of the Hankel matrix as a product involving the coefficient matrix of the orthogonal polynomials. We consider a more general situation than Berg. The sequences of generalized Catalan numbers that we consider use the rational parameter $q/p$ instead of a natural number parameter $\alpha$, and also use an additional scaling factor $p^{2n}$. We also consider subsequences starting at an offset from the beginning of the sequence. We still use his general approach, based on the coefficient matrix of orthogonal polynomials. Our proofs will use computer algebra to show that the polynomials for the sequence $g^{(q/p)}$ are orthogonal. 

\section{Orthogonal polynomials for G}
We define a lower triangular matrix $L$, whose rows are the coefficients of orthogonal polynomials for the bilinear form defined by $G$. The coefficients are based on the coefficients of shifted Jacobi polynomials as defined by Hetyei \cite[equation (4)]{HETYEI}, and are then scaled by powers of $p^2$.
\begin{definition}
\label{CATBERTOPLOWERTRI}
Let $a$, $p$, and $q$ be as in definition \ref{GENPATHANKEL}. Define $L=L^{(a,q/p)}$ to be the lower triangular matrix with 
\begin{equation}
L_{n,k} = (-1)^{n+k}p^{2k}\binom{n+k+a+q/p-1}{k}\binom{n+a}{k+a}. 
\end{equation}
\end{definition}

To show that $L$ is a coefficient matrix of orthogonal polynomials for the bilinear form defined by $G$, we show that the product $LGL^T$ is a diagonal matrix.

We show this in two steps. First, we show that the rows of $L$ satisfy a three term recurrence. Second, we prove that the rows are orthogonal with respect to the bilinear form defined by $G$.

\section{Three Term Recurrence and Orthogonality}
We state and prove a three term recurrence relation for the rows of L.
\begin{lemma}
\label{CATBERTTHREETERM}
Let 
\begin{equation}
\alpha_n=-\frac{(n+a+1)(q+np)(q+2np++ap+3p)}{(n+2)(q+np+ap+p)(q+2np+ap+p)},
\end{equation}
 let 
\begin{equation}
\beta_n=-\frac{(q+2np+ap+2p)(2nq+aq+3q+2n^2p+2anp+4np+(a+1)^2p)}{(n+2)(q+np+ap+p)(q+2np+ap+p)},
\end{equation}
and let 
\begin{equation}
\gamma_n=\frac{p(q+2np+ap+2p)(q+2np+ap+3p)}{(n+2)(q+np+ap+p)}.
\end{equation}
Then 
\begin{equation}
\label{THREETERMEQN}
\alpha_nL(n,k)+\beta_nL(n+1,k)+\gamma_nL(n+1,k-1)=L(n+2,k).
\end{equation}
\end{lemma}
The proof of Lemma \ref{CATBERTTHREETERM} follows from calculations performed with the wxMaxima computer algebra system \cite{WXMAX,MAX}. The wxMaxima script to perform the calculations is in the ancillary file `CatalanJacobiTHREETERM.wxm' included with this article.

Next we show that the rows of $L$ are orthogonal with respect to the bilinear form defined by $G$, by showing that $LGL^T$ is a diagonal matrix. Specifically, we show that $LG$ is upper triangular, and that implies that $LGL^T$ is diagonal.

\begin{lemma}
\label{CATBERTORTHOGLEM}
Let $n$ be an integer with $n \ge 2$. The matrices $L$ and $G$ satisfy
\[\sum_{k=0}^n L_{n,k}G_{k,n-1}=0\] and
\[\sum_{k=0}^n L_{n,k}G_{k,n-2}=0\].
\end{lemma}
\begin{proof}
The wxMaxima script to verify these equations is in the ancillary file `CatalanJacobiORTHOGONAL.wxm' included with this article.
\end{proof}

\begin{theorem}
\label{CATBERTOPTHEOREM}
The product $LGL^T$ is a diagonal matrix.
\end{theorem}
\begin{proof}
The matrices $L$ and $G$ satisfy
$\sum_{k=0}^n L_{n,k}G_{k,m}=0$ for all $m$ with $0\le m < n$,
by Lemma \ref{CATBERTORTHOGLEM}, Lemma \ref{CATBERTTHREETERM}, and mathematical induction. This says that $LG$ is an upper triangular matrix. Since $G$ is symmetric, this implies that $GL^T$ is lower triangular. As $L$ is lower triangular, this implies that $LGL^T$ is both upper triangular and lower triangular, and thus $LGL^T$ is diagonal.
\end{proof}

\section{The Diagonal of the Product}
For the forthcoming proofs, we will need to know the value of the diagonal entries of $LGL^T$.

\begin{definition}
\label{NORMDEFIN}
Let $a$, $p$, and $q$ be as in definition \ref{GENPATHANKEL}.
Define $N$ to be the diagonal matrix with \[N^{(a,q/p)}_{n,n} = \frac{p^{2a}(2np+ap+q)\binom{n+a+q/p-1}{a}}{q\binom{n+a}{a}}.\]
\end{definition}

\begin{theorem}
\label{CATBERTNORM}
The matrices $L$, $G$, and $N$ satisfy 
\begin{equation} 
\label{CATBERTNORMEQN}
LGL^T=N^{-1}.
\end{equation}
\end{theorem}

We prove this theorem by a calculation with the Zeilberger algorithm as implemented in wxMaxima. By Lemma \ref{CATBERTORTHOGLEM}, the entries of the diagonal of $LGL^T$ are given by $L_{n,n}\sum_{k=0}^n L_{n,k}G_{k,n}$, so the theorem follows from the summation identity
\begin{equation}
    \sum_{k=0}^n L_{n,k}G_{k,n}L_{n,n} = N^{-1}_{n,n}.
\end{equation}
For an explanation of how the Zeilberger algorithm proves such a summation identity, see Section 2.3 of \cite{AeqB}, particularly the itemized list before Example 2.3.1. For the computer calculation, we divide the summand by the conjectured sum value $N^{-1}_{n,n}$, as in step 2 of the referenced list. In our case, the output from the function `Zeilberger' contains a first order recurrence with non-constant terms. As the coefficients of the recurrence are negatives of each other, it immediately shows that the sum over $k$ is constant as a function of $n$. The wxMaxima script to perform the calculations is in the ancillary file `CatalanJacobiDIAGONAL.wxm' included with this article.

\section{Lucas's Theorem and Variations}
For the next proofs, we need to use Lucas's theorem on binomial coefficients modulo a prime number, and some variations.

\begin{lemma}
\label{LUCASLEM}
Let $p$ be a prime number, and let $n$ and $k$ satisfy $n \equiv 0 \pmod{p}$, and $k \not\equiv 0 \pmod{p}$. Then $\binom{n}{k} \equiv 0 \pmod{p}$.
\end{lemma}
Lemma \ref{LUCASLEM} is a simple consequence of Lucas's theorem \cite[Theorem 1]{BAILEY}.

We next consider three variations of Lucas's theorem that involve binomial coefficients with non-integer parameters.
\begin{lemma}
\label{LUCASVAR1}
Let $q$ be a prime number, let $p$ be an integer with $\gcd(q,p)=1$, let $n$ be an integer such that 
$p^2n+pq \equiv 0 \pmod{q}$, 
and let $k$ be a positive integer such that $k \not\equiv 0 \pmod{q}$. 
Then $p^{2k}\binom{n+q/p}{k} \equiv 0 \pmod{q}$.
\end{lemma}

\begin{proof}
We may express 
\begin{equation}
\label{LUCASVAR1EQN}
p^{2k}\binom{n+q/p}{k}=
\frac{\prod_{i=0}^{k-1}\bigl(p^2(n-i)+pq\bigr)}{\prod_{i=0}^{k-1}(i+1)}.
\end{equation}
By induction on $k$, it follows that the numerator of equation (\ref{LUCASVAR1})
is always divisible by more powers of $q$ than the denominator, except when 
$k \equiv 0 \pmod{q}.$.
\end{proof}

\begin{lemma}
\label{LUCASVAR2}
Let $q=-3$, let $p=2$, 
let $n$ satisfy $n \equiv 2 \pmod{3}$,
and let $k$ satisfy 
$k \equiv 2 \pmod{3}$. 
Then $4^k\binom{n+k-3/2}{k} \equiv 0 \pmod{3}$.
\end{lemma}

\begin{proof}
We may express 
\begin{equation}
\label{LUCASVAR2EQN}
4^k\binom{n+k-3/2}{k} = 4^k\binom{n+k-1-3/2}{k}+4^k\binom{n+k-1-3/2}{k-1}.\end{equation}
Now by Lemma \ref{LUCASVAR1}, 
$4^k\binom{n+k-1-3/2}{k}$ and $4^k\binom{n+k-1-3/2}{k-1}$
are both divisible by $3$, therefore
$4^k\binom{n+k-3/2}{k}$ is also divisible by $3$.
\end{proof}

\begin{lemma}
\label{LUCASVAR3}
Let $p$ be a prime number, let $q$ be an integer with $\gcd(q,p)=1$, let $n$ be an integer, and let $k$ be a positive integer.
Then $p^{2k}\binom{n+q/p}{k} \equiv 0 \pmod{p}$.
\end{lemma}

\begin{proof}
We express $p^{2k}\binom{n+q/p}{k}$ as 
\[\prod_{i=0}^k \frac{p^2n+pq}{i+1}. \]
In this expression, the numerator is divisible by $p^k$, while the largest power of $p$ that divides the denominator is at most $p^{(k-1)/(p-1)}$ by Legendre's formula \cite[Equation (1.2)]{STRAUB}.
\end{proof}

\section{Main Theorems}

\begin{theorem}
\label{MAINTHEOREM}
Let $a$, $p$, $q$, and $G^{(a,q/p)}(n)$ be as in Definitions \ref{GENPATSEQ} and \ref{GENPATHANKEL}. Then 

\noindent
(1) if $q=\pm 1$, the entries of the inverse of the matrix $G^{(a,q/p)}(n)$ are integers.

\noindent
(2) if $q=\pm 2$, the entries of the inverse of the matrix $G^{(a,q/p)}(n)$ are integers.
\end{theorem}

From equation \eqref{CATBERTNORMEQN}, it follows that
the $n\times n$ upper left submatrix $G(n)$ satisfies
\begin{equation}
\label{GEQN}
G(n) = L(n)^{-1}N(n)^{-1}(L(n)^T)^{-1}.
\end{equation} 
Thus for the $n\times n$ upper left submatrix $G(n)$ we have
\begin{equation}
\label{GINVEQUATION}
G(n)^{-1}=L(n)^{T}N(n)L(n). 
\end{equation} 
We now introduce an alternate expression for the product $N(n)L(n)$ that will allow us to remove the binomial coefficient term in the denominators of the entries of $N(n)$.

\begin{definition}
\label{MKDEFIN}
Define the diagonal matrix $M$ by 
\[M^{(a,q/p)}_{n,n} = \frac{(2np+ap+q)}{q}.\]
Define the lower triangular matrix $K$ by 
\[
K^{(a,q/p)}_{n,k} = (-1)^{n+k}p^{2(k+a)}\binom{n+k+a+q/p-1}{k+a}\binom{n}{k}. 
\]
\end{definition}

\begin{lemma}
\label{MKLEMMA}
The matrices $K$, $L$, $M$, and $N$ satisfy
\begin{equation}
\label{NLMKEQN}
NL = MK.
\end{equation}
\end{lemma}

\begin{proof}
Since $N$ and $M$ are diagonal, it is sufficient to prove that 
\begin{equation}
\label{NLMKEQN2}
N_{n,n}L_{n,k} = M_{n,n}K_{n,k}.
\end{equation}
The $(n,k)$ entry of $NL$ is given by  
\begin{align}
    N_{n,n}&L_{n,k} 
    \\ = & 
    \frac{p^{2a}(2np+ap+q)\binom{n+a+q/p-1}{a}}{q\binom{n+a}{a}}
    (-1)^{n+k}p^{2k}\tbinom{n+k+a+q/p-1}{k}\tbinom{n+a}{k+a}
    \\ = &
    (-1)^{n+k}p^{2(k+a)}(2np+ap+q)\frac{\binom{n+a+q/p-1}{a}\binom{n+k+a+q/p-1}{k}\binom{n+a}{k+a}}{q\binom{n+a}{a}}
    \\ = &
    (-1)^{n+k}p^{2(k+a)}(2np+ap+q)\frac{\binom{n+k+a+q/p-1}{k+a}\binom{n}{k}}{q}
    \\ = &
    \frac{2np+ap+q}{q}(-1)^{n+k}p^{2(k+a)}\tbinom{n+k+a+q/p-1}{k+a}\tbinom{n}{k}
    \\ = &
    M_{n,n}K_{n,k},
\end{align}
and the last expression is the $(n,k)$ entry of $MK$.  
\end{proof}
From Lemma \ref{MKLEMMA} we can rewrite the factorization of $G(m)^{-1}_{n,k}$ in equation (\ref{GINVEQUATION}) as
\begin{equation}
\label{LMKEQN}
G(m)^{-1}=L(m)^TM(m)K(m).
\end{equation}
We can express the $(n,k)$ entry of $G(m)^{-1}$ as
\begin{equation}
\label{LMKEQN3}
G(m)^{-1}_{n,k}=\sum_{i=n}^{m-1} L(m)_{i,n}M(m)_{i,i}K(m)_{i,k}.
\end{equation}
By symmetry and equation \eqref{NLMKEQN2}, we have  
\begin{equation}
\label{LMKEQN4}
L(m)_{i,n}M(m)_{i,i}K(m)_{i,k}=K(m)_{i,n}M(m)_{i,i}L(m)_{i,k}.
\end{equation}
Writing out the full expression of the products in equation \eqref{LMKEQN4}, the left side expands as
\begin{equation}
\label{LMKPRODEQN}
\begin{split}
    L_{i,n}M_{i,i}K_{i,k} &\\
    = &(-1)^{n+k}(p^2)^{n+k+a}
{\textstyle  \frac{(2ip+ap+q)}{q}
            \binom{i+n+q/p-1}{n}\binom{i+a}{n+a} 
             \binom{i+k+q/p-1}{k+a}\binom{i}{k}}, \\
\intertext{and the right side expands as}
    K_{i,n}M_{i,i}L_{i,k} &\\
    = &(-1)^{n+k}(p^2)^{n+k+a}
{\textstyle  \frac{(2ip+ap+q)}{q} 
            \binom{i+n+q/p-1}{n+a}\binom{i}{n} 
             \binom{i+k+q/p-1}{k}\binom{i+a}{k+a}}.
\end{split}    
\end{equation}

\begin{proof} (of Theorem \ref{MAINTHEOREM}.)
Using equations \eqref{LMKEQN}-\eqref{LMKPRODEQN} to prove the theorem we have to show that $q$ divides one of the terms of the numerator of a right-hand side of equation \eqref{LMKPRODEQN}. Case (1) is obvious. For case (2), we show that $2$ divides the numerator of one side of equation \eqref{LMKEQN4}, in particular, that $2$ divides either the numerator of $M_{i,i}$, $K_{i,k}$, or $L_{i,k}.$
If $a$ is even, then $2$ divides the numerator of $M(m)_{i,i}$, so we may assume that $a$ is odd. Now if $i$ and $k$ have different parity, then one of $\binom{i}{k}$ or $\binom{i+a}{k+a}$ is divisible by $2$ by Lemma \ref{LUCASLEM}. Hence $2$ divides either $K_{i,k}$ or $L_{i,k},$ respectively.

Now we may assume that $i$ and $k$ have the same parity, still assuming that $a$ is odd. Now either $k$ or $k+a$ is odd. By Lemma \ref{LUCASVAR1} this implies that $2$ divides either 
$p^{2k}\binom{i+k+a+q/p-1}{k}$ or $p^{2(k+a)}\binom{i+k+a+q/p-1}{k+a}$, hence $2$ divides either $L_{i,k}$ or $K_{i,k},$ respectively.

Thus in every case, we have shown that $2$ divides the summand in equation \eqref{LMKEQN3}, so $2$ divides $G^{-1}_{n,k}$, and hence $G(m)^{-1}$ is an integer matrix when $q=\pm 2$.
\end{proof}

\begin{theorem}
\label{MAINTHEOREM2}
Let $a$, $p$, $q$, and $G^{(a,q/p)}(n)$ be as in Definitions \ref{GENPATSEQ} and \ref{GENPATHANKEL}. Then the entries of the inverse of the matrix $\frac{1}{q}G^{(a,q/p)}(n)$ are integers.
\end{theorem}
\begin{proof}  
By Lemma \ref{MKLEMMA}, it follows that
\[ \bigl( \frac{1}{q}G(m)\bigr)^{-1} 
= qG(m)^{-1} = L(m)^T\bigl(qM(n)\bigr)K(m).\]
Since $qM$, $L$, and $K$ are integer matrices, it follows that
\[\bigl( \frac{1}{q}G(m)\bigr)^{-1}\] is an integer matrix.
\end{proof}

We now use equation \eqref{LMKEQN} to express the determinant of 
$\bigl(G(n)^{(a,q/p)}\bigr)^{-1}$ as a product of the diagonal elements of $M{(a,q/p)}$, $L{(a,q/p)}$, and $K{(a,q/p)}$, omitting the factors that are identically equal to $1$.

\begin{theorem}
\label{DETTHM}
The determinant of the matrix $\bigl(G(n)^{(a,q/p)}\bigr)^{-1}$ is given by
\begin{multline}
\label{DETEQN1}
\det \Bigl(\bigl(G(n)^{(a,q/p)}\bigr)^{-1}\Bigr)
\\ =
\prod_{k=0}^{n-1}
(p^2)^{2k+a} \frac{2kp+ap+q}{q}\binom{2k+a+q/p-1}{k}\binom{2k+a+q/p-1}{k+a},
\end{multline}
and the determinant of the matrix $\bigl(\frac{1}{q}G(n)^{(a,q/p)}\bigr)^{-1}$ is given by
\begin{multline}
\label{DETEQN2}
\det \Bigl(\bigl(\frac{1}{q}G(n)^{(a,q/p)}\bigr)^{-1}\Bigr)
\\ =
\prod_{k=0}^{n-1}
(p^2)^{2k+a} (2kp+ap+q)\binom{2k+a+q/p-1}{k}\binom{2k+a+q/p-1}{k+a}.
\end{multline}
\end{theorem}

\begin{proof}
Equation \eqref{LMKEQN} implies that the determinant of \(\bigl(G(n)^{(a,q/p)}\bigr)^{-1}\)
is the product of the diagonals of $L(n)^{(a,q/p)}$, $M(n)^{(a,q/p)}$, and $K(n)^{(a,q/p)}$. Similarly, the determinant of \(\bigl(\frac{1}{q}G(n)^{(a,q/p)}\bigr)^{-1}\)
is the product of the diagonals of $L(n)^{(a,q/p)}$, $qM(n)^{(a,q/p)}$, and $K(n)^{(a,q/p)}$.
\end{proof}

\section{The Catbert Matrix}
The Catbert matrix is the Hankel matrix of reciprocals of the Catalan numbers. The sequence \(g^{(-3/2)} = 1,-2,-2,-4,-10,\ldots\) is related to the Catalan numbers; \(g^{(-3/2)}\) is the sequence of Catalan numbers, multiplied by $-2$, and prepended by \(1\). We express the Catbert matrix in terms of a matrix $G^{(a,q/p)}$ by using the offset $1$ and multiplying by $-2$.

\begin{definition}
\label{CATBERTDEFIN}
The {\em Catbert matrix} is the matrix \( C \) given by 
\( C=-2G^{(1,-3/2)}=-2G \).
\end{definition}
It follows that
\( L^{(1,-3/2)} \)
is a coefficient matrix of orthogonal polynomials for the Catbert matrix \( C \). 
\begin{theorem}

\label{CATBERTTHM}
The inverse of the Catbert matrix has integer entries.
\end{theorem}

\begin{proof}
Let $L = L^{(1,-3/2)}$, $K = K^{(1,-3/2)}$, and $M = M^{(1,-3/2)}$. By definitions \ref{CATBERTOPLOWERTRI}, \ref{NORMDEFIN}, and \ref{MKDEFIN}, we have
\begin{equation}
    L_{i,k} = (-1)^{i+k}4^{k}\binom{i+k-3/2}{k}\binom{i+1}{k+1},
\end{equation}
\begin{equation}
    K_{i,k} = (-1)^{i+k}4^{k+1}\binom{i+k-3/2}{k+1}\binom{i}{k}, 
\end{equation}
and
\begin{equation}
    M_{i,i} = \frac{4i-1}{3}. 
\end{equation}
Now \( C(m)^{-1} = \frac{1}{2}\bigl(G^{(1,-3/2)}\bigr)^{-1} = \frac{1}{2}L(m)^TM(m)K(m) \).
We need to show that both $2$ and $3$ divde some factor of the numerator of 
\[G(m)^{-1}=L(m)^TM(m)K(m)=K(m)^TM(m)L(m).\] Lemma \ref{LUCASVAR3} implies that $2$ divides $K_{i,k}.$

To prove divisibility by $3$, we use equations \eqref{LMKEQN3} and \eqref{LMKEQN4} as in the proof of Theorem \ref{MAINTHEOREM}. If $i\equiv 1 \pmod{3}$ then $3$ divides $4i-1$, and so $3$ divides the numerator of $M_{i,i}$.
If $i\equiv 0 \pmod{3}$ and $k \not\equiv 0 \pmod{3}$, then $3$ divides $\binom{i}{k}$  by Lemma \ref{LUCASLEM}, and thus $3$ divides $K_{i,k}$. 
If $i\equiv 2 \pmod{3}$ and $k \not\equiv 2 \pmod{3}$, then $3$ divides $\binom{i+1}{k+1}$  by Lemma \ref{LUCASLEM}, and thus $3$ divides $L_{i,k}$. 
If $i\equiv 0 \pmod{3}$ and $k \equiv 0 \pmod{3}$, then $3$ divides $4^{k}\binom{i+k-3/2}{k+1}$  by Lemma \ref{LUCASVAR1}, and thus $3$ divides $K_{i,k}$. 
If $i\equiv 2 \pmod{3}$ and $k \equiv 2 \pmod{3}$, then $3$ divides $4^{k}\binom{i+k-3/2}{k}$  by Lemma \ref{LUCASVAR2}, and thus $3$ divides $L_{i,k}$. 
So for all possible values of $i\pmod{3}$ and $k\pmod{3}$, we have shown that $3$ divides the numerator of \(L(m)^TM(m)K(m)\). Hence \(C^{-1}\), the inverse of the Catbert matrix, has integer entries.
\end{proof}

We can now express the determinant of the inverse Catbert matrix $C(n)^{-1}$ as a product of the diagonal elements of $M$, $L$, and $K$, omitting the factors that are identically equal to $1$.

\begin{corollary}
The determinant of the inverse Catbert matrix $C(n)^{-1}$ is given by
\begin{equation}
\label{CATBERTDET}
det\bigl(C(n)^{-1}\bigr)=\prod_{k=0}^{n-1}
4^{2k+1}\frac{4k-1}{6}\binom{2k-3/2}{k}\binom{2k-3/2}{k+1}.
\end{equation}
\end{corollary}

The sequence of determinants of the inverse Catbert matrix is sequence \seqnum{A296056}
in the {\em On-Line Encyclopedia of Integer Sequences} \cite{OEIS}.

\end{document}